\documentclass[11pt]{article}

\usepackage[T1]{fontenc}
\usepackage{lmodern}
\usepackage{amsmath,amssymb,amsthm}
\usepackage[margin=1in]{geometry}
\usepackage[hidelinks]{hyperref}

\hypersetup{
  pdftitle={The Maximum Permanent of a Stochastic Matrix of Bounded Rank},
  pdfauthor={Yair Lavi}
}

\newtheorem{theorem}{Theorem}[section]
\newtheorem{proposition}[theorem]{Proposition}
\newtheorem{lemma}[theorem]{Lemma}
\newtheorem{corollary}[theorem]{Corollary}
\theoremstyle{remark}
\newtheorem*{remark}{Remark}

\title{The Maximum Permanent of a Stochastic Matrix of Bounded Rank}
\author{Yair Lavi}
\date{21 August 2026}

\begin{document}
\maketitle

\begin{abstract}

Let \(A\) be a stochastic \(n\times n\) matrix
with \(\operatorname{rank}A\leq k\), where \(1\leq k\leq n\). Write
\(n=qk+s\), where \(0\leq s<k\). The author conjectured in 2018 that

\[
 \operatorname{per}A\leq
 \left(\frac{q!}{q^q}\right)^{k-s}
 \left(\frac{(q+1)!}{(q+1)^{q+1}}\right)^s,
\]

with equality if and only if

\[
 A=P\left(J_q^{\oplus(k-s)}\oplus J_{q+1}^{\oplus s}\right)Q,
\]

where \(P,Q\) are permutation matrices and \(J_t\) is the \(t\times t\)
matrix with every entry \(1/t\). We prove this conjecture in full.
\end{abstract}



\section{Introduction and main theorem}

For an \(n\times n\) matrix \(A=(a_{ij})\), its permanent is

\[
 \operatorname{per}A
 =\sum_{\sigma\in S_n}\prod_{i=1}^n a_{i,\sigma(i)}.
 \tag{1.1}
\]

A nonnegative matrix is \textbf{row stochastic}, \textbf{column stochastic}, or
\textbf{doubly stochastic} when all its row sums, all its column sums, or both,
respectively, are one. Write \(\Omega_n\) for the set of doubly stochastic
\(n\times n\) matrices.

Throughout, \(\operatorname{rank}\) denotes ordinary matrix rank.

The elementary bound \(\operatorname{per}A\leq1\) for a row- or
column-stochastic matrix is attained exactly by permutation matrices, all of
which have full rank.

For a positive integer \(t\), let \(\mathbf1_t\in\mathbb R^t\) denote the
column vector whose entries are all one, and set

\[
 J_t=\frac1t\mathbf1_t\mathbf1_t^{\mathsf T}.
\]

When its length is clear, we write simply \(\mathbf1\).

Given \(1\leq k\leq n\), write

\[
 n=qk+s,\qquad 0\leq s<k,                                  \tag{1.2}
\]

and let \(\rho_{n,k}\) denote the balanced \(k\)-part partition consisting
of \(k-s\) parts \(q\) and \(s\) parts \(q+1\). For a partition
\(\rho=(t_1,\ldots,t_k)\), put

\[
 J_\rho=J_{t_1}\oplus\cdots\oplus J_{t_k},
 \qquad
 C_{n,k}=\operatorname{per}J_{\rho_{n,k}}
 =\prod_{t\in\rho_{n,k}}\frac{t!}{t^t}.                    \tag{1.3}
\]

The author conjectured in 2018 \cite{lavi2018} that every nonnegative row- or
column-stochastic \(n\times n\) matrix \(A\) with
\(\operatorname{rank}A\leq k\) satisfies

\[
 \operatorname{per}A\leq C_{n,k}
 =\left(\frac{q!}{q^q}\right)^{k-s}
  \left(\frac{(q+1)!}{(q+1)^{q+1}}\right)^s.              \tag{1.4}
\]

The author also conjectured that equality holds exactly when
\(A=PJ_{\rho_{n,k}}Q\) for permutation matrices \(P,Q\). We prove both
claims for every order and rank ceiling.

\begin{theorem}
Let
\(1\leq k\leq n\), and let \(A\geq0\) be an \(n\times n\) matrix with
\(\operatorname{rank}A\leq k\).

\par\smallskip
\noindent\textup{1.}\enspace If \(A\) is row stochastic or column stochastic,
then

\[
 \boxed{\operatorname{per}A\leq C_{n,k}.}                \tag{1.5}
\]

Equality holds exactly when

\[
 A=PJ_{\rho_{n,k}}Q                                      \tag{1.6}
\]

for permutation matrices \(P,Q\).

\par\smallskip
\noindent\textup{2.}\enspace For an arbitrary nonnegative \(A\), let \(r_i\)
and \(c_j\) be its row and column sums, and put

\[
 R(A)=\prod_{i=1}^n r_i,
 \qquad
 C(A)=\prod_{j=1}^n c_j.
\]

Then

\[
 \boxed{
 \operatorname{per}A
 \leq C_{n,k}\min\{R(A),C(A)\}.}                         \tag{1.7}
\]

For completeness, the equality cases in (1.7) are as follows. A zero row or
column always gives equality. Otherwise all line sums are positive; define

\[
 A^r=\operatorname{diag}(r_1^{-1},\ldots,r_n^{-1})A,
 \qquad
 A^c=A\operatorname{diag}(c_1^{-1},\ldots,c_n^{-1}).        \tag{1.8}
\]

If \(R(A)<C(A)\), equality holds exactly when
\(A^r=PJ_{\rho_{n,k}}Q\). If \(C(A)<R(A)\), it holds exactly when
\(A^c=PJ_{\rho_{n,k}}Q\). If \(R(A)=C(A)>0\), equality holds exactly when
both normalizations in (1.8) have the displayed balanced-block form;
equivalently, it is enough that either one does.
\end{theorem}

\begin{remark}
Parts (1) and (2) are equivalent: Lemma 2.1 proves the
equivalence of the inequalities, and the same normalizations identify their
equality cases. We state both because the proof proceeds by simultaneous
induction on the stochastic and homogeneous formulations.
\end{remark}

Every matrix in (1.6) has rank exactly \(k\), so the rank ceiling is sharp.

In the singular specialization \(k=n-1\), the balanced partition is
\((2,1,\ldots,1)\), and Theorem 1.1 gives the following particularly simple
consequence conjectured in \cite{lavi2018}.

\begin{corollary}
Let \(n\geq2\). If \(A\) is a
singular row- or column-stochastic \(n\times n\) matrix, then

\[
 \operatorname{per}A\leq\frac12.                           \tag{1.9}
\]

Equality holds exactly when \(A=P(J_2\oplus I_{n-2})Q\) for permutation
matrices \(P,Q\).
\end{corollary}

The entire upper half of the rank range also has a simple closed form.

\begin{corollary}
Let \(n/2<k\leq n\). If \(A\)
is a row- or column-stochastic \(n\times n\) matrix with
\(\operatorname{rank}A\leq k\), then

\[
 \operatorname{per}A\leq\frac{1}{2^{n-k}}.                \tag{1.10}
\]

Equality holds exactly when

\[
 A=P\bigl(J_2^{\oplus(n-k)}\oplus I_{2k-n}\bigr)Q
\]

for permutation matrices \(P,Q\). For \(n\geq3\), Corollary 1.2 is the
case \(k=n-1\).
\end{corollary}

When the balanced partition has equal parts, the general constant likewise
simplifies.

\begin{corollary}
Let \(n=qk\), where \(q,k\geq1\). If
\(A\) is a row- or column-stochastic \(n\times n\) matrix with
\(\operatorname{rank}A\leq k\), then

\[
 \operatorname{per}A
 \leq \left(\frac{q!}{q^q}\right)^k.                      \tag{1.11}
\]

Equality holds exactly when \(A=P(J_q^{\oplus k})Q\) for permutation
matrices \(P,Q\).
\end{corollary}

\begin{remark}
Each of these
corollaries has a corresponding form for arbitrary nonnegative matrices:
multiply the right-hand side of its displayed inequality by
\(\min\{R(A),C(A)\}\). The equality cases, including the zero-line cases,
are precisely those specified after Theorem 1.1 through the normalizations
in (1.8).
\end{remark}

Rank-sensitive upper bounds for permanents of doubly stochastic matrices
go back at least to Marcus and Minc
\cite[Corollaries~3.1 and~3.3]{marcusminc1965}. They
proved that, for \(A\in\Omega_n\),

\[
 \operatorname{per}A\leq
 \sqrt{\frac{\operatorname{rank}A}{n}},                    \tag{1.12}
\]

and, if \(A\) is normal, the stronger bound

\[
 \operatorname{per}A\leq\frac{\operatorname{rank}A}{n}.  \tag{1.13}
\]

Equality in (1.12) requires a permutation matrix. The normal inequality is
strict except for a permutation matrix and the exceptional matrix \(J_2\).
These are earlier permanent bounds expressed directly in terms of ordinary
rank, but they do not determine the maximum at a prescribed rank or its
equality cases. For example, at the singular ceiling \(k=n-1\), the general
Marcus--Minc bound (1.12) gives \(\sqrt{(n-1)/n}\), while their stronger
bound (1.13) for normal matrices gives \((n-1)/n\). For \(n\geq3\), both
values exceed the sharp value \(1/2\) from Corollary 1.2.

The proof has three main ingredients. First, a rank-safe scaling reduction
shows that it is enough to consider doubly stochastic matrices. Second,
Tverberg's theorem and a reciprocal linear-programming duality imply that a
doubly stochastic matrix of rank at most \(k\) has a row whose entries are
all at most \(1/\lceil n/k\rceil\). Third, expansion along this row transfers
the problem to order \(n-1\). Sections 2--4 establish these ingredients and
the sharp value. Section 5 proves the equality classification.

\section{Normalization and doubly stochastic reduction}

\subsection{Homogeneous normalization}

We begin with two elementary observations that will also organize the
induction.

\begin{lemma}
Fix \(n,k\) and a constant
\(\Gamma>0\). The following assertions are equivalent.

\par\smallskip
\noindent\textup{1.}\enspace Every nonnegative row- or column-stochastic
\(n\times n\) matrix \(A\) with \(\operatorname{rank}A\leq k\) satisfies
\(\operatorname{per}A\leq\Gamma\).

\par\smallskip
\noindent\textup{2.}\enspace Every nonnegative \(n\times n\) matrix \(A\)
with \(\operatorname{rank}A\leq k\) satisfies

\[
 \operatorname{per}A
 \leq\Gamma\min\{R(A),C(A)\}.                          \tag{2.1}
\]
\end{lemma}

\begin{proof}
Assume (1). If some row sum or column sum is zero, nonnegativity
makes the corresponding row or column zero, so both sides of (2.1) vanish.
Otherwise all line sums are positive. Left normalization gives

\[
 \operatorname{per}A=R(A)\operatorname{per}A^r,
 \qquad
 \operatorname{rank}A^r=\operatorname{rank}A.              \tag{2.2}
\]

Similarly, right normalization gives

\[
 \operatorname{per}A=C(A)\operatorname{per}A^c.            \tag{2.3}
\]

Applying (1) to \(A^r\) and \(A^c\) yields (2.1), proving (2).

Conversely, assume (2). If \(A\) is row stochastic, then \(R(A)=1\); if it
is column stochastic, then \(C(A)=1\). In either case
\(\min\{R(A),C(A)\}\leq1\), so (2.1) gives
\(\operatorname{per}A\leq\Gamma\). This proves (1).
\end{proof}

The same identities prove the homogeneous equality statement in Theorem
1.1 once the stochastic equality cases are known.

\begin{lemma}
Every nonnegative row- or
column-stochastic \(n\times n\) matrix \(A\) satisfies

\[
 \operatorname{per}A\leq1,                                 \tag{2.4}
\]

with equality exactly at the permutation matrices.
\end{lemma}

\begin{proof}
For a row-stochastic matrix,

\[
 1=\prod_{i=1}^n\left(\sum_{j=1}^n a_{ij}\right)
 =\sum_{f:[n]\to[n]}\prod_{i=1}^n a_{i,f(i)}.              \tag{2.5}
\]

The permanent is the subsum indexed by bijections, proving the bound. If
equality holds, every noninjective term is zero. If two different rows had
positive entries in a common column, choosing those entries and arbitrary
positive entries in the remaining rows would give a positive noninjective
term. Thus the nonempty row supports are pairwise disjoint subsets of an
\(n\)-element set. Each is a singleton, and row stochasticity makes \(A\) a
permutation matrix. The column-stochastic case follows by transposition.
\end{proof}

\subsection{A rank-safe Sinkhorn reduction}

We use the standard total-support form of the Sinkhorn--Knopp scaling
theorem \cite{sinkhornknopp1967}: a nonzero square nonnegative matrix can be
scaled by positive
diagonal matrices to a doubly stochastic matrix if every positive entry lies
on a positive diagonal.

For a nonnegative matrix \(M=(m_{ij})\) and \(\sigma\in S_n\), call

\[
 \prod_{i=1}^n m_{i,\sigma(i)}
\]

the \textbf{permanent monomial} indexed by \(\sigma\). We call the monomial
positive when its value is positive, and in that case call \(\sigma\) a
\textbf{positive permutation} for \(M\). Equivalently, \(\sigma\) is positive
for \(M\) exactly when \(m_{i,\sigma(i)}>0\) for every \(i\).

\begin{proposition}
Let \(A\geq0\) be row
stochastic and suppose \(\operatorname{per}A>0\). There is a
\(B\in\Omega_n\) such that

\[
 \operatorname{rank}B\leq\operatorname{rank}A,
 \qquad
 \operatorname{per}A\leq\operatorname{per}B.               \tag{2.6}
\]

Furthermore, if \(A\) is not doubly stochastic, \(B\) can be chosen so that
\(\operatorname{per}A<\operatorname{per}B\).
\end{proposition}

\begin{proof}
Choose a positive permanent monomial, indexed by \(\sigma\), and
permute columns so that the main diagonal of \(A\) is positive. We will undo
this permutation on both \(A\) and the constructed matrix \(B\) at the end.
On the vertex set \([n]\),
form the directed graph having an arc \(i\to j\) exactly when \(a_{ij}>0\). An entry
\(a_{ij}>0\) belongs to a positive permanent monomial exactly when the arc
\(i\to j\) lies on a directed cycle. Indeed, a positive permutation
decomposes into such cycles. Conversely, a simple directed cycle containing
\(i\to j\) (with a loop \(i\to i\) allowed), together with the positive
main-diagonal entries outside that cycle, is a positive permutation.

Call these entries \textbf{active}, and let \(A^\circ\) retain the active entries
and set all other entries to zero. Recall that two vertices belong to the
same strongly connected component when each is reachable from the other by
a directed path. Mutual reachability is an equivalence relation and hence
partitions \([n]\) into components \(V_1,\ldots,V_t\). A positive arc
\(i\to j\) lies on a directed cycle if and only if \(i\) and \(j\) belong
to the same component: a cycle gives paths in both directions, while a path
\(j\leadsto i\) closes the arc \(i\to j\) to a cycle. Thus \(A^\circ\)
retains exactly the positive entries whose row and column indices lie in the
same component.

Form the directed graph whose vertices are the components and whose arrows
record the positive entries of \(A\) between distinct components. This
component graph is acyclic. Indeed, a directed cycle of components would
make all of its vertices mutually reachable and hence part of a single
strongly connected component. We may therefore order the components so that
every arrow between distinct components points from an earlier component to
a later one; this is a topological ordering. Applying this same order to the
rows and columns makes \(A\) block upper triangular. The matrix \(A^\circ\)
deletes all entries between distinct components and is therefore block
diagonal. Each diagonal block is square because the same set \(V_h\) indexes
its rows and columns.

Let \(A_h=A[V_h,V_h]\) be the \(h\)-th diagonal block and put
\(d_h=\operatorname{rank}A_h\). Choose subsets \(R_h,C_h\subseteq V_h\),
each of size \(d_h\), such that
\(\det A_h[R_h,C_h]\ne0\). Set \(R=\bigcup_hR_h\) and
\(C=\bigcup_hC_h\). In the topological component order, the square submatrix
\(A[R,C]\) is block upper triangular with diagonal blocks
\(A_h[R_h,C_h]\). Consequently,

\[
 \det A[R,C]=\prod_{h=1}^t\det A_h[R_h,C_h]\ne0.
\]

Thus \(A\) has a nonsingular minor of size \(\sum_h d_h\), and hence

\[
 \operatorname{rank}A^\circ=\sum_{h=1}^t d_h
 \leq\operatorname{rank}A.                                 \tag{2.7}
\]

Every nonzero permanent monomial of \(A\) uses only active entries, so

\[
 \operatorname{per}A^\circ=\operatorname{per}A.            \tag{2.8}
\]

The matrix \(A^\circ\) has total support. By Sinkhorn--Knopp, there are
positive diagonal matrices \(D,E\) such that

\[
 B=DA^\circ E\in\Omega_n.                                  \tag{2.9}
\]

It remains to check the direction of the permanent inequality. For a
nonnegative \(n\times n\) matrix \(M\) and
\(x=(x_1,\ldots,x_n)\in\mathbb R_{>0}^n\), define

\[
 p_M(x)=\prod_{i=1}^n(Mx)_i,
 \qquad
 \operatorname{Cap}(p_M)
 =\inf_{x\in\mathbb R_{>0}^n}
   \frac{p_M(x)}{x_1\cdots x_n}.                         \tag{2.10}
\]

If \(B\) is doubly stochastic, weighted AM--GM gives

\[
 \prod_i(Bx)_i
 \geq\prod_i\prod_jx_j^{b_{ij}}
 =\prod_jx_j,
\]

and equality holds at \(x=\mathbf1\). Thus
\(\operatorname{Cap}(p_B)=1\). If

\[
 \gamma=(\det D)(\det E),
\]

then a change of variables in (2.10) gives

\[
 1=\operatorname{Cap}(p_B)
 =\gamma\operatorname{Cap}(p_{A^\circ}).                  \tag{2.11}
\]

Every row sum of \(A^\circ\) is at most the corresponding row sum of \(A\),
which is one. Therefore

\[
 \operatorname{Cap}(p_{A^\circ})
 \leq p_{A^\circ}(\mathbf1)\leq1,                         \tag{2.12}
\]

so \(\gamma\geq1\). Positive diagonal scaling and (2.7)--(2.9) now give

\[
 \operatorname{per}B
 =\gamma\operatorname{per}A^\circ
 \geq\operatorname{per}A,
 \qquad
 \operatorname{rank}B=\operatorname{rank}A^\circ
 \leq\operatorname{rank}A.                                \tag{2.13}
\]

This proves the asserted domination. It remains to show that, when the
original \(A\) is not doubly stochastic, the constructed matrix \(B\)
satisfies \(\operatorname{per}A<\operatorname{per}B\). We prove the
contrapositive: suppose \(\operatorname{per}A=\operatorname{per}B\). Since
the permanent is positive, \(\gamma=1\). Equations (2.11)--(2.12) force
every row sum of \(A^\circ\)
to be one. Thus \(A=A^\circ\). Consider the convex function

\[
 F(z)=\sum_{i=1}^n
 \log\!\left(\sum_{j=1}^n a_{ij}e^{z_j}\right)
 -\sum_{j=1}^n z_j.                                       \tag{2.14}
\]

Putting \(x_j=e^{z_j}\) in (2.10), and using \(A=A^\circ\) and \(\gamma=1\),
gives \(\inf F=\log\operatorname{Cap}(p_A)=-\log\gamma=0\). Since
\(F(0)=0\), zero minimizes \(F\), and

\nopagebreak[4]
\[
 0=\frac{\partial F}{\partial z_j}(0)
 =\sum_{i=1}^n a_{ij}-1                                   \tag{2.15}
\]

for every \(j\). The column sums of \(A\) are one. Since \(A\) was row
stochastic, it is therefore doubly stochastic, as required. Undoing the
initial column permutation preserves row stochasticity, the property of
being doubly stochastic, permanent, and rank, and therefore gives the
asserted matrix for the original \(A\). Thus, when the original \(A\) is not
doubly stochastic, the inequality in (2.6) is strict.
\end{proof}

Consequently, a sharp inequality and equality classification on
\(\Omega_n\) automatically gives the corresponding row- and
column-stochastic result without increasing ordinary rank.

\section{Tverberg theory and a diffuse row}

\subsection{A diagonal-rank theorem}

We use the classical affine Tverberg theorem \cite{tverberg1966}: any

\[
 (r-1)(d+1)+1                                               \tag{3.1}
\]

points, not necessarily distinct, in an affine space of dimension at most
\(d\) can be partitioned into \(r\) nonempty sets whose convex hulls have a
common point. If more points are present, the extras may be distributed among
the parts.

\begin{theorem}
Let \(P=(p_{ij})\geq0\)
be an \(n\times n\) row-stochastic matrix with
\(\operatorname{rank}P\leq k\), and put

\[
 m=\left\lceil\frac nk\right\rceil.
\]

Then

\[
 \boxed{\min_i p_{ii}\leq\frac1m.}                         \tag{3.2}
\]

The constant is sharp.
\end{theorem}

\begin{proof}
Write \(p_i\) for row \(i\) of \(P\). The rows lie in the
rowspace of \(P\), and all lie in the affine hyperplane with coordinate sum
one. Their affine dimension is therefore at most \(k-1\). Since

\[
 (m-1)k<n,
 \qquad\text{hence}\qquad
 n\geq(m-1)k+1,                                             \tag{3.3}
\]

Tverberg's theorem gives a partition

\[
 [n]=S_1\sqcup\cdots\sqcup S_m                             \tag{3.4}
\]

and a probability vector \(x\) belonging to every
\(\operatorname{conv}\{p_i:i\in S_t\}\). Choose coefficients
\(\lambda_i^{(t)}\geq0\), summing to one over \(i\in S_t\), such that

\[
 x=\sum_{i\in S_t}\lambda_i^{(t)}p_i.                      \tag{3.5}
\]

If every \(p_{ii}>1/m\), nonnegativity would give

\[
 \begin{aligned}
 x(S_t)
 &=\sum_{j\in S_t}x_j\\
 &=\sum_{i\in S_t}\lambda_i^{(t)}
   \sum_{j\in S_t}p_{ij}\\
 &\geq\sum_{i\in S_t}\lambda_i^{(t)}p_{ii}
 >\frac1m.
 \end{aligned}                                             \tag{3.6}
\]

Summing over the partition (3.4) would yield \(1>1\), a contradiction.
This proves (3.2).

For sharpness, take \(P=J_{\rho_{n,k}}\). Its rank is \(k\), and its
smallest diagonal entry is the reciprocal of its largest block size, namely
\(1/m\).
\end{proof}

\begin{remark}
A
row-stochastic matrix has spectral radius at most one. At most
\(\operatorname{rank}P\) of its eigenvalues, counted with algebraic
multiplicity, are nonzero. Taking real parts in the trace identity therefore
gives

\[
 \operatorname{tr}P\leq\operatorname{rank}P\leq k.
\]

Consequently \(\min_i p_{ii}\leq k/n\). When \(k\mid n\), this is already
the value \(1/\lceil n/k\rceil\) in Theorem 3.1. At the singular endpoint
\(k=n-1\), the same value \(1/2\) also follows from strict diagonal
dominance: if every \(p_{ii}>1/2\), then
\(p_{ii}>\sum_{j\ne i}p_{ij}\) in every row, forcing \(P\) to be
nonsingular. The genuinely stronger numerical content of Theorem 3.1 is
the nondivisible range, where

\[
 \frac1{\lceil n/k\rceil}<\frac{k}{n}.
\]

The shared indexing of rows and coordinates is essential in this argument.
No column-sum condition is assumed.
\end{remark}

\subsection{A rounded centroid theorem}

We now convert Theorem 3.1 into the geometric statement needed for rows of a
doubly stochastic matrix.

\begin{theorem}
Let \(1\leq k\leq n\), and let
\(v_1,\ldots,v_n\) be listed points in an affine space of dimension at most
\(k-1\). Put

\[
 \bar v=\frac1n\sum_{i=1}^n v_i,
 \qquad
 m=\left\lceil\frac nk\right\rceil.
\]

There is a convex representation

\[
 \bar v=\sum_{i=1}^n\lambda_i v_i,
 \qquad
 \lambda_i\geq0,
 \qquad
 \sum_i\lambda_i=1,                                       \tag{3.7}
\]

with at most \(k\) nonzero coefficients and

\[
 \boxed{\max_i\lambda_i\geq\frac mn.}                     \tag{3.8}
\]

The constant is sharp.
\end{theorem}

\begin{proof}
Translate so that \(\bar v=0\), choose coordinates in
\(\mathbb R^d\), where \(d\leq k-1\), and let \(V\) be the matrix with
columns \(v_i\). Define

\[
 W=\operatorname{rowspan}
 \begin{pmatrix}\mathbf1^{\mathsf T}\\V\end{pmatrix},
 \qquad
 L=W^\perp.                                                 \tag{3.9}
\]

The polytope of representations of the origin is

\[
 \mathcal P
 =\{p\geq0:\mathbf1^{\mathsf T}p=1,\ Vp=0\}
 =\left(\frac1n\mathbf1+L\right)\cap\Delta_{n-1}.         \tag{3.10}
\]

For each coordinate, set

\[
 \alpha_i=\max_{p\in\mathcal P}p_i,
 \qquad
 \mu_i=\max_{q\in W\cap\Delta_{n-1}}q_i.                \tag{3.11}
\]

We first prove the reciprocal identity

\[
 \boxed{\alpha_i\mu_i=\frac1n.}                            \tag{3.12}
\]

Write \(p=(\mathbf1+y)/n\) in (3.10), and put
\(s_i=n\alpha_i-1\). Linear-programming duality \cite{schrijver1986} gives

\[
 \begin{aligned}
 s_i
 &=\max\{y_i:y\in L,\ y_j\geq-1\ \text{for every }j\}\\
 &=\min\{\mathbf1^{\mathsf T}z:
          z\geq0,\ e_i+z\in W\}.                          
 \end{aligned}                                             \tag{3.13}
\]

The primal is feasible because \(y=0\), and it is bounded above: since
\(\mathbf1\in W\), every \(y\in L\) has coordinate sum zero, so the lower
bounds \(y_j\geq-1\) imply \(y_i\leq n-1\). Finite-dimensional strong
duality therefore applies, and the displayed dual minimum is attained. Every
optimal \(z\) has \(z_i=0\): if \(z_i>0\), replacing \(e_i+z\) by
\((e_i+z)/(1+z_i)\) and then subtracting \(e_i\) gives a nonnegative feasible
dual vector with smaller coordinate sum. Consequently

\[
 q=\frac{e_i+z}{1+s_i}\in W\cap\Delta_{n-1},
 \qquad
 q_i=\frac1{1+s_i}=\frac1{n\alpha_i}.                      \tag{3.14}
\]

This proves \(\mu_i\geq1/(n\alpha_i)\). Conversely, if
\(p\in\mathcal P\) and \(q\in W\cap\Delta_{n-1}\), then

\[
 p^{\mathsf T}q=\frac1n                                   \tag{3.15}
\]

because \(p-\mathbf1/n\in L=W^\perp\). Nonnegativity gives
\(p_iq_i\leq1/n\); maximizing both coordinates proves the reverse
inequality in (3.12).

Suppose that \(\alpha_i<m/n\) for every \(i\). Then (3.12) would give
\(\mu_i>1/m\) for every \(i\). For each \(i\), choose
\(q^{(i)}\in W\cap\Delta_{n-1}\) with \(q_i^{(i)}=\mu_i\), and stack these
vectors as the rows of a matrix \(Q\). Then

\[
 Q\geq0,
 \qquad Q\mathbf1=\mathbf1,
 \qquad \operatorname{rank}Q\leq\dim W\leq k,
 \qquad q_{ii}>\frac1m,                                   \tag{3.16}
\]

contradicting Theorem 3.1. Hence some \(\alpha_i\geq m/n\). A maximizing
point may be chosen at a vertex of \(\mathcal P\). By the standard
basic-feasible-solution support bound \cite{schrijver1986}, it has at most
\(\dim W\leq k\) positive coordinates, proving the support assertion.

For sharpness, take \(k\) affinely independent point types with the balanced
multiplicities in \(\rho_{n,k}\). Every representation of their centroid has
fixed total weight \(t/n\) on a type of multiplicity \(t\). Thus no
coefficient can exceed \(m/n\), while concentrating each type weight on one
listed copy attains \(m/n\).
\end{proof}

\begin{remark}
Ordinary Carath\'eodory gives a
representation of \(\bar v\) on at most \(k\) listed points; since its
coefficients sum to one, one coefficient is at least \(1/k\). Theorem 3.2
improves this to

\[
 \frac{\lceil n/k\rceil}{n}\geq\frac1k,
\]

with strict improvement exactly when \(k\nmid n\).
\end{remark}

\subsection{The diffuse-row consequence}

\begin{corollary}
If \(A\in\Omega_n\) and
\(\operatorname{rank}A\leq k\), then some row \(a_i\) satisfies

\[
 \boxed{\max_j a_{ij}\leq
 \frac1{\lceil n/k\rceil}.}                                \tag{3.17}
\]
\end{corollary}

\begin{proof}
As in the proof of Theorem 3.1, the rows of \(A\) have affine dimension at
most \(k-1\). Their centroid is the uniform probability vector

\[
 u=\frac1n\mathbf1^{\mathsf T}                              \tag{3.18}
\]

because the column sums are one. By Theorem 3.2,

\[
 u=\sum_r\lambda_r a_{i_r}
\]

with some \(\lambda_{r_0}\geq m/n\), where
\(m=\lceil n/k\rceil\). For every coordinate \(j\), nonnegativity gives

\[
 \frac1n=u_j
 \geq\lambda_{r_0}a_{i_{r_0}j}
 \geq\frac mn a_{i_{r_0}j},                               \tag{3.19}
\]

and hence \(a_{i_{r_0}j}\leq1/m\).
\end{proof}

\section{Released-row induction and the sharp value}

For \(t\geq1\), write

\[
 f(t)=\frac{t!}{t^t}.
\]

When \(n>k\) and \(m=\lceil n/k\rceil\), deleting one element from a largest
part of \(\rho_{n,k}\) produces \(\rho_{n-1,k}\). Therefore

\[
 \frac{C_{n,k}}{C_{n-1,k}}
 =\frac{f(m)}{f(m-1)}
 =\left(1-\frac1m\right)^{m-1}.                            \tag{4.1}
\]

The same factor arises from the following sharp probability inequality.

\begin{lemma}
Let
\(a=(a_1,\ldots,a_N)\) satisfy

\[
 a_j\geq0,
 \qquad \sum_j a_j=1,
 \qquad a_j\leq\frac1m,
\]

where \(m\geq2\) is an integer. Then

\[
 H(a):=\sum_{j=1}^N a_j\prod_{\ell\ne j}(1-a_\ell)
 \leq\left(1-\frac1m\right)^{m-1}.                        \tag{4.2}
\]

Equality holds exactly when, up to permutation,

\[
 a=(\underbrace{1/m,\ldots,1/m}_{m},0,\ldots,0).           \tag{4.3}
\]
\end{lemma}

\begin{proof}
The feasible set is nonempty only if \(N\geq m\), and it is compact, so \(H\)
has a maximizer. Fix one. Probabilistically, \(H(a)\) is the chance of exactly
one success among independent Bernoulli variables with success probabilities
\(a_j\). Since the coordinates sum to one and each is at most \(1/m\), at
least \(m\geq2\) are positive. After relabeling, suppose two of them are
\(a_1,a_2\). Put \(s=a_1+a_2\) and \(v=a_1a_2\), and let \(G_0,G_1\) be the
probabilities of zero and one success among the Bernoulli variables with
probabilities \(a_3,\ldots,a_N\). Recombining the selected pair gives

\[
 H=(1-s+v)G_1+(s-2v)G_0
 =\bigl((1-s)G_1+sG_0\bigr)+v(G_1-2G_0).                 \tag{4.4}
\]

During a fixed-sum perturbation of the selected pair, the remaining
coordinates and \(s\) are fixed. Thus the parenthesized term in (4.4) is
constant and only \(v\) varies.

Since \(a_r\leq1/m<1\), the probability
\(G_0=\prod_{r=3}^N(1-a_r)\) is positive. The mutually exclusive choices for
the unique successful remaining variable therefore give

\[
 \frac{G_1}{G_0}
 =\sum_{r=3}^N\frac{a_r}{1-a_r}
 \leq\frac{m}{m-1}\sum_{r=3}^N a_r
 =\frac{m}{m-1}(1-s)<2.                                   \tag{4.5}
\]

The last inequality follows from \(s>0\) and \(m/(m-1)\leq2\). Hence
\(G_1-2G_0<0\), so (4.4) shows that \(H\) strictly decreases with
\(v=a_1a_2\) when \(s=a_1+a_2\) is fixed.

Put \(u=1/m\). Keeping \(s\) and the remaining coordinates fixed, write the
selected pair as \((z,s-z)\). Its feasible interval is

\[
 \max\{0,s-u\}\leq z\leq\min\{u,s\}.
\]

The product \(z(s-z)\) is strictly concave. If both selected coordinates lie
strictly between \(0\) and \(u\), moving \(z\) to a suitable endpoint of
this interval strictly decreases their product and therefore strictly increases
\(H\). At an endpoint, one coordinate equals \(0\) when \(s\leq u\), or one
coordinate equals \(u\) when \(s\geq u\). It follows that a maximizer has at
most one coordinate strictly between \(0\) and \(u\).

If a maximizer had one such coordinate \(z\) and \(r\) coordinates equal to
\(u\), then

\[
 z=1-\frac rm=\frac{m-r}{m}.
\]

The condition \(0<z<1/m\) would imply \(0<m-r<1\), which is impossible
because \(m-r\) is an integer. Thus every maximizer has exactly \(m\)
coordinates equal to \(1/m\) and all other coordinates zero. Substitution
gives the value \((1-1/m)^{m-1}\), proving both (4.2) and its equality
statement.
\end{proof}

We can now prove the value assertion in Theorem 1.1.

\begin{proof}[Proof of the inequality in Theorem~1.1]
Fix \(k\), and induct
simultaneously over \(n\geq k\) on the stochastic theorem and its homogeneous
form. When \(n=k\), \(C_{k,k}=1\), so Lemmas 2.1 and 2.2 give the result.

Let \(n>k\), and assume the theorem at order \(n-1\) with the same rank
ceiling \(k\). Let \(A\geq0\) be row stochastic with
\(\operatorname{rank}A\leq k\). The case \(\operatorname{per}A=0\) is
immediate. Otherwise, Proposition 2.3 shows that it is enough to consider
\(A\in\Omega_n\). Corollary 3.3 supplies a row
\(a=(a_1,\ldots,a_n)\) with \(a_j\leq1/m\), where
\(m=\lceil n/k\rceil\). Expanding the permanent along that row gives

\[
 \operatorname{per}A
 =\sum_{j=1}^n a_j\operatorname{per}A(i\mid j),             \tag{4.6}
\]

where \(A(i\mid j)\) denotes the submatrix obtained by deleting the selected
row and column \(j\). Its rank is at most \(k\). For every remaining
column \(\ell\ne j\), the column sum of this submatrix is \(1-a_\ell\). The
order-\((n-1)\) homogeneous induction hypothesis therefore gives

\[
 \operatorname{per}A(i\mid j)
 \leq C_{n-1,k}\prod_{\ell\ne j}(1-a_\ell).                \tag{4.7}
\]

Combining (4.6), (4.7), Lemma 4.1, and (4.1), we obtain

\[
 \begin{aligned}
 \operatorname{per}A
 &\leq C_{n-1,k}H(a)\\
 &\leq C_{n-1,k}\left(1-\frac1m\right)^{m-1}\\
 &=C_{n,k}.
 \end{aligned}                                             \tag{4.8}
\]

This proves the row-stochastic statement at order \(n\); transposition gives
the column-stochastic statement, and Lemma 2.1 then gives the homogeneous
statement. The simultaneous induction is complete.
\end{proof}

\section{Equality cases}

It remains to prove that equality in the stochastic theorem forces the
balanced blocks. Proposition 2.3 already shows that a row-stochastic equality
matrix must be doubly stochastic. We therefore work in \(\Omega_n\).

\begin{theorem}
If
\(A\in\Omega_n\), \(\operatorname{rank}A\leq k\), and

\[
 \operatorname{per}A=C_{n,k},                              \tag{5.1}
\]

then

\[
 A=PJ_{\rho_{n,k}}Q                                        \tag{5.2}
\]

for permutation matrices \(P,Q\).
\end{theorem}

\begin{proof}
Fix \(k\) and induct on \(n\geq k\), together with the value
proof. At \(n=k\), equality in Lemma 2.2 makes \(A\) a permutation matrix,
which is (5.2) because \(J_{\rho_{k,k}}=I_k\).

Let \(n>k\), put \(m=\lceil n/k\rceil\), and choose an index \(i\) such that
the row

\[
 a:=A_{i,*}=(a_1,\ldots,a_n)
\]

is supplied by Corollary 3.3. Equality throughout (4.8), together with the
equality case of Lemma 4.1, forces this row to be

\nopagebreak[4]
\[
 a=\frac1m\mathbf1_S,
 \qquad |S|=m.                                              \tag{5.3}
\]

Here \(S\subseteq[n]\), and \(\mathbf1_S\in\mathbb R^n\) is its indicator
vector: its \(\ell\)-th coordinate is one for \(\ell\in S\) and zero
otherwise.

Fix \(j\in S\), set \(T=S\setminus\{j\}\), and let
\(M=A(i\mid j)\). Since \(a_j>0\), equality in the sum (4.6)--(4.7) forces

\[
 \operatorname{per}M
 =C_{n-1,k}\prod_{\ell\ne j}(1-a_\ell).                   \tag{5.4}
\]

The columns of \(M\) indexed by \(T\) have sum \((m-1)/m\), and all its
other columns have sum one. Scale the \(T\)-columns by \(m/(m-1)\), and call
the resulting matrix \(D\). Then \(D\) is column stochastic,

\[
 \operatorname{rank}D\leq k,
 \qquad
 \operatorname{per}D=C_{n-1,k}.                            \tag{5.5}
\]

By the transpose of Proposition 2.3 and the order-\((n-1)\) value bound, \(D\)
must be doubly stochastic: otherwise the strict part of the proposition would
produce a matrix in \(\Omega_{n-1}\) of rank at most \(k\) and permanent
strictly greater than \(C_{n-1,k}\). The equality-case induction hypothesis
now yields

\[
 D=P'J_{\rho_{n-1,k}}Q'.                                   \tag{5.6}
\]

In particular, \(D\) has rank exactly \(k\).

Apply the same \(T\)-column scaling to all of \(A\), and then delete only
column \(j\). After placing the selected row first, the resulting
\(n\times(n-1)\) matrix has the form

\[
 \begin{pmatrix}r\\D\end{pmatrix},
 \qquad
 r=\frac1{m-1}\mathbf1_T.                                  \tag{5.7}
\]

Here \(\mathbf1_T\in\mathbb R^{n-1}\) is the indicator vector of \(T\) in
the remaining column coordinate set \([n]\setminus\{j\}\).

Column scaling and deletion cannot increase rank, so the rectangular matrix
has rank at most \(k\). Since \(D\) has rank \(k\), we must have

\[
 r\in\operatorname{rowspan}D.                              \tag{5.8}
\]

The rowspace of a direct sum of \(J\)-blocks consists exactly of the vectors
that are constant on each column block. Hence (5.7)--(5.8) show that \(T\)
is a union of complete column blocks in (5.6).

Since \((m-1)k\leq n-1<mk\), every part of \(\rho_{n-1,k}\) has size
\(m-1\) or \(m\), and at least one has size \(m-1\). Since \(|T|=m-1\),
the set \(T\) is exactly one \((m-1)\)-column block. Let \(R\) be its
associated row block.
Before undoing the scaling, this block is \(J_{m-1}\). Undoing the scaling
changes all its entries from \(1/(m-1)\) to \(1/m\). Thus, in the matrix
obtained from \(A\) by deleting column \(j\), the rows in \(R\), as well as
the selected row \(i\), have row sum \((m-1)/m\), while all other rows have
row sum one. Row stochasticity of \(A\) forces column \(j\) to equal \(1/m\)
on \(R\cup\{i\}\) and zero elsewhere. Together with (5.3), the selected
row, column \(j\), and the block
\(R\times T\) form one \(J_m\) block. Every other block in (5.6) is
unchanged.

Replacing one \((m-1)\)-part of \(\rho_{n-1,k}\) by an \(m\)-part gives
exactly \(\rho_{n,k}\). This proves (5.2).
\end{proof}

Conversely, each \(PJ_{\rho_{n,k}}Q\) is doubly stochastic, has rank \(k\),
and has permanent \(C_{n,k}\). Theorem 5.1 and Proposition 2.3 therefore
give the exact row-stochastic equality class; transposition gives the column
class.

Finally, suppose \(A\geq0\) has positive line sums. Equations (2.2)--(2.3)
show that equality in the smaller of the two homogeneous bounds is exactly
equality for the corresponding stochastic normalization. If
\(R(A)=C(A)\), equality in one normalization is automatically equality in
the other. Together with the zero-row and zero-column cases, this is exactly
the homogeneous classification stated after Theorem 1.1. The proof of the
theorem is complete.

\section*{Acknowledgments}

The proof of this conjecture was carried out by GPT-5.6-sol under the guidance
of the author. The author has reviewed the resulting proof arguments.
Responsibility for the final text rests with the author.


\begin{thebibliography}{9}

\bibitem{lavi2018}
Y. Lavi,
\newblock The permanent and diagonal products on the set of nonnegative
matrices with bounded rank,
\newblock arXiv:1808.00016v1 (2018).
\newblock \href{https://doi.org/10.48550/arXiv.1808.00016}{doi:10.48550/arXiv.1808.00016}.

\bibitem{marcusminc1965}
M. Marcus and H. Minc,
\newblock Generalized matrix functions,
\newblock \emph{Trans. Amer. Math. Soc.} \textbf{116} (1965), 316--329.
\newblock \href{https://doi.org/10.1090/S0002-9947-1965-0194445-9}{doi:10.1090/S0002-9947-1965-0194445-9}.

\bibitem{sinkhornknopp1967}
R. Sinkhorn and P. Knopp,
\newblock Concerning nonnegative matrices and doubly stochastic matrices,
\newblock \emph{Pacific J. Math.} \textbf{21} (1967), no.~2, 343--348.
\newblock \href{https://doi.org/10.2140/pjm.1967.21.343}{doi:10.2140/pjm.1967.21.343}.

\bibitem{tverberg1966}
H. Tverberg,
\newblock A generalization of Radon's theorem,
\newblock \emph{J. London Math. Soc.} \textbf{41} (1966), 123--128.
\newblock \href{https://doi.org/10.1112/jlms/s1-41.1.123}{doi:10.1112/jlms/s1-41.1.123}.

\bibitem{schrijver1986}
A. Schrijver,
\newblock \emph{Theory of Linear and Integer Programming},
\newblock John Wiley \& Sons, Chichester, 1986.

\end{thebibliography}
\end{document}